\documentclass[11pt]{amsart}
\usepackage{geometry}                
\usepackage{graphicx}
\usepackage{amssymb}
\usepackage{epstopdf}
\usepackage{hyperref}
\usepackage{setspace}
\usepackage{amsmath}
\usepackage{amsthm}
\usepackage[all]{xy}

\usepackage{xparse}
\usepackage{tikz}
\usetikzlibrary{matrix,backgrounds}
\pgfdeclarelayer{myback}
\pgfsetlayers{myback,background,main}
\tikzset{mycolor/.style     = {line width=1bp,color=#1}}
\tikzset{myfillcolor/.style = {rounded corners,draw,fill=#1}}
\NewDocumentCommand{\fhighlight}{O{blue!20} m m}{\draw[myfillcolor=#1] (#2.north west)rectangle (#3.south east);}

\theoremstyle{definition}
\newtheorem{theorem}{Theorem}[section]
\newtheorem*{theorem*}{Theorem}

\newtheorem{corollary}[theorem]{Corollary}
\newtheorem{proposition}[theorem]{Proposition}
\newtheorem{lemma}[theorem]{Lemma}

\newtheorem{definition}[theorem]{Definition}

\newtheorem{remark}[theorem]{Remark}

\newtheorem{example}[theorem]{Example}

\newtheorem{step}{Step}

\newcommand{\pp}{\ensuremath{\mathbb{P}}}

\newcommand{\zc}{\ensuremath{\mathcal{Z}}}
\newcommand{\Q}{\ensuremath{\mathbb{P}^{1}\times \mathbb{P}^{1}}}

\newcommand{\QP}{\ensuremath{\left(\begin{array}{llll} Q_{0}& Q_{1}&Q_{2}&Q_{3} \\ P_{0}& P_{1}&P_{2}&P_{3} \\ \end{array} \right)}}

\title[Implicitization of tensor product surfaces in the presence of basepoints]{Implicitization of tensor product surfaces in the presence of a generic set of  basepoints}
\author{Eliana Duarte}
       \address{Department of mathematics, University of Illinois, Urbana, IL 61801 }
       \email{emduart2@illinois.edu}
\begin{document}
\onehalfspacing
\maketitle
     
\begin{abstract}
Given a $4$-dimensional vector subspace $U=\{ f_{0},\ldots,f_{3}\}$ of 
$H^{0}(\Q,\mathcal{O}(a,b))$, a tensor product surface, denoted by $X_{U}$, is the closure of
the image of the rational map $\lambda_{U}:\Q \dashrightarrow \pp^{3}$ determined by $U$.
These surfaces arise in geometric modeling and in this context it is useful to know the implicit
equation of $X_{U}$ in $\pp^{3}$. In this paper
we show that if $U\subseteq H^{0}(\Q,\mathcal{O}(a,1))$ has a finite set of $r$ basepoints in
generic position, then the implicit equation of $X_{U}$ is determined
by two syzygies of $I_{U}=\langle f_{0},\ldots,f_{3} \rangle$ in bidegrees $\left( a-\lceil\frac{r}{2}\rceil,0 \right)$ and
$\left( a-\lfloor\frac{r}{2}\rfloor,0 \right)$. This result is proved by understanding the geometry of the basepoints of $U$ in $\Q$. The proof techniques for the main theorem also apply when $U$ is basepoint free.

\end{abstract}
\tableofcontents
 
\section{Introduction} \label{introduction}

Given  a parameterized curve or surface in  projective space, the implicitization
problem consists on finding the equations whose vanishing is the closed
image of the given parameterization. The implicitization problem has been of increasing interest to 
commutative
algebraists and algebraic geometers due to its applications in Computer Aided Geometric Design(CAGD).  
In this context, knowing the implicit equation of the curve or surface is important to perform elementary 
operations with these objects. For example describing the curve of intersection of two surfaces, or the 
points of intersection of a curve and a surface. Equally important is the problem of testing if a given point 
in the codomain lies in the image of a parameterization. Using the implicit equations of a 
curve or surface to perform these operations is computationally and theoretically more efficient than only 
working with the parameterizations. For this reason, there is great interest in finding faster algorithms that
compute implicit equations of parameterized curves and surfaces.

  There are several general
  techniques  to solve the implicitization problem for parameterized surfaces. These include
  Gr\"{o}bner bases, resultants and syzygies. To summarize, Gr\"{o}bner bases algorithms provide a 
  straightforward theoretical approach that in practice tends to be very slow. 
  Resultants provide a more convenient representation of the implicit equation 
  as a determinant of a matrix but they fail for parameterizations with \emph{basepoints}. In contrast, since 
  their first appearance in the work of Sederberg and Chen \cite{sedchen95} and Cox
  \cite{CurvesSurfsAndSyzygies}, syzygies have provided faster methods to obtain implicit equations that work 
  for more general 
  parameterizations. We elaborate more on these methods and compare timings for some examples in Section
  \ref{implicitization}.

    We now focus  on the implicitization problem for tensor product surfaces.  Given a
 $4$ dimensional vector space $U=\{f_{0},\ldots,f_{3}\}\subset H^{0}(\Q,\mathcal{O}(a,b))$ with 
 $\gcd{(f_{0},\ldots,f_{3})}=1$, we  obtain a rational map $\lambda_{U}:\Q \dashrightarrow \pp^{3}$. A
 \emph{tensor product} surface is the closure of the image of  $\lambda_{U}$ and  is denoted $X_{U}$.
    In this paper we use syzygy techniques to obtain the implicit equation of a tensor product surface 
  $X_{U}$. One of the main tools of this approach is the approximation complex $\zc$ defined in Section
  \ref{implicitization}. Approximation complexes were first introduced by Herzog, Simis and Vasconcelos 
  \cite{hsv} 
  to study the defining equations of Rees algebras of ideals. An introduction to the use of these
  complexes to obtain implicit equations of parameterized hypersurfaces and the connection with Rees algebras
   can be found in Chardin's \cite{Chardin} work. These techniques were generalized for multigraded 
   hypersurfaces such as tensor product surfaces by Botbol \cite{botbol}. In Section \ref{implicitization} 
   we use Botbol's results to obtain the implicit 
  equation of $X_{U}$ via $\zc$.

   In short, to find the implicit equation of $X_{U}$ using $\zc$
   we fix a basis and find the matrices representing the maps of $\zc$ in a suitable degree $\nu$.
   Finally, the determinant of the complex $\zc_{\nu}$ is a power of the implicit equation of $X_{U}$. 
   It is important to point out that for most practical purposes in CAGD, knowing the matrix of the
   map $d_{1}:(\zc_{\nu})_{1}\to (\zc_{\nu})_{0}$  is sufficient to perform the aforementioned 
   elementary operations with curves and surfaces, for example as shown by 
   Buse and Luu Ba \cite{surfsIntBuse} for the
   intersection of two surfaces. 
   The matrix $d_{1}$ is known as a representation
   matrix for $X_{U}$. Representation matrices are generically of full rank and the $\gcd$ of its
   maximal minors is equal to a power of the implicit equation of $X_{U}$. These matrices have been studied
    by Botbol and Dickenstein \cite{botdick} and  Botbol, Buse and Chardin 
   \cite{botbuscha} among others.
   
   The connection of $\zc$ with the syzygies of $U=\{f_{0},\ldots,f_{3}\}$ comes from the observation that
   $\zc_{1}=\mathrm{Syz}(f_{0},\ldots,f_{3})\otimes S$, where $S:=k[X,Y,Z,W]$ is the coordinate ring of 
   $\pp^{3}$. Thus to obtain the matrices
   in $\zc_{\nu}$ we are led to the computation of $\mathrm{Syz}(f_{0},\ldots,f_{3})$.  
   Let $I_{U}=\langle f_{0},
   \ldots,f_{3}\rangle$ be the ideal generated by $U$ inside the total coordinate ring of $\Q$. 
   The vanishing locus $\mathbb{V}(I_{U})$ inside $\Q$ is referred
   to as the set of basepoints of $U$ and is denoted by $X$. If $X\neq \emptyset$ we say $U$ has basepoints, 
   otherwise we say $U$ is basepoint free. We also refer to $X$ as the base locus of $\lambda_{U}$. 
   The content of this 
   paper follows the ideas from Duarte and Schenck \cite{DS} and of Schenck, Seceleanu and Validashti 
   \cite{SSTPS} in which
   the goal is to understand the syzygies that determine $\zc_{\nu}$.
   We focus
   on tensor product surfaces
   $X_{U}$ such that $U\subset H^{0}(\Q,\mathcal{O}(a,b))$ has $b=1$ and such that $X$ is a finite set of 
   generic points. 
   
  Tensor product surfaces for which $b=1$ are also known as
   rational ruled surfaces in the literature. These surfaces have been studied before by Chen, Zheng and Sederberg~\cite{chenZhengSed} and by Dohm
   \cite{dohmRRS}. Chen, Zheng and Sederberg show that the implicit equation of a generically injective parameterization
   can be obtained from a submodule of the module of syzygies of $U$ and show that this module is of rank 
   two. Dohm generalizes these results  for a parameterization of any degree and provides an
   algorithm to reparameterize the given rational ruled surface. In this work we study the syzygies of
   $U$  based on the geometry of its 
   base locus $X$. This approach gives the exact degrees of the syzygies that determine $\zc_{\nu}$ and answers 
   the question raised by Chen, Cox and Liu \cite{ChenCoxLiu} of what can be said about the degrees of the syzygies that 
   determine the implicit equation. The statement
   of the main theorem is the following:
   
   \begin{theorem}\label{mainth}
     Let $(I_{X})_{(a,1)}\subset H^{0}(a,1)$ be the $k$-vector space of
     forms of bidegree $(a,1)$ that vanish at a generic set $X$ of $r$ points
     in $\Q$. Take $U=\{f_{0},\ldots,f_{3}\}$ to be a generic $4$-dimensional
     vector subspace of $(I_{X})_{(a,1)}$ and $\lambda_{U}:\Q\dashrightarrow \pp^{3}$
     the rational map determined by $U$. Then the first map of the
     approximation complex $\mathcal{Z}$ in bidegree $\nu=(2a-1,0)$, 
     $(d_{1})_{\nu}$, is determined by the syzygies of $(f_{0},\ldots,f_{3})$
     in bidegrees 
     $ \left(a- \left\lfloor \frac{r}{2}\right\rfloor,0 \right),\;\; \left(a-\left\lceil\frac{r}{2}\right 
     \rceil,
    0   \right)$. The map $(d_{1})_{\nu}$
     completely determines the approximation complex $\mathcal{Z}$ in degree
     $\nu$ from which we compute the implicit equation of $X_{U}$.
  \end{theorem}
  Using the description of
  some of the minimal generators the ideal $I_{X}$ associated to a generic set of points $X$ in $\Q$ given
  by Van Tuyl \cite{van2005defining} we are able to understand exactly the degrees of the syzygies of $(f_{0},\ldots,f_{3})$
   that determine the implicit equation of $X_{U}$.
   The subsequent sections of this paper are organized as follows. In Section \ref{preliminaries}
   we introduce  basic notation and several tools to study points in $\Q$. In Section 
   \ref{idealOfPoints} we use results of Van Tuyl \cite{van2005defining} to describe the minimal
   generators of the bihomogeneous ideal $I_{X}$ associated to a set $X$ of points in $\Q$ in generic
   position.
   In Section \ref{syzygiesOfIdeal} we use the results from Section \ref{idealOfPoints}
   to describe the syzygies of the ideal $I_{U}$ where $U\subset (I_{X})_{(a,1)}$ is a generic four 
   dimensional vector subspace. Finally in Section \ref{implicitization} we use the results
   from the previous sections to prove the main theorem for implicitization of tensor product 
   surfaces in the presence of a generic set of basepoints by using the
   complex $\zc$ and give some examples. Throughout this paper $k$ will denote an algebraically 
   closed field of characteristic zero.


\section{Points in $\Q$} \label{preliminaries}
The idea to prove Theorem \ref{mainth} is  to study the geometry of the base locus of $U$ 
in $\Q$, to do this we focus on the ideals of $R$ that correspond to generic points
in $\Q$. In this section we describe such ideals and state a theorem of Van Tuyl for their bigraded
Hilbert functions.
\subsection{Notation}
 Let $R=\bigoplus_{0\leq a, 0\leq b}H^{0}(\pp^{1}\times \pp^{1},\mathcal{O}(a,b))$ be the total
 coordinate ring of $\pp^{1}\times \pp^{1}$. For shorter notation we write $H^{0}(a,b)$ for
 $H^{0}(\pp^{1}\times \pp^{1},\mathcal{O}(a,b))$.
 The ring $R$ is a bigraded $k$-algebra by taking
 $R_{(a,b)}=H^{0}(a,b)$ and $R$ is generated as a $k$-algebra
 by $H^{0}(1,0)$ and $H^{0}(0,1)$.
 Note that $\dim H^{0}(1,0) = \dim H^{0}(0,1)=2 $. If $\{s,t\}$ is a basis for $H^{0}(1,0)$ and $\{u,v\}$ a basis for $H^{0}(0,1)$, then
 $R\cong k[s,t]\otimes k[u,v]$ with grading given by $\deg s,t=(1,0)$ and $\deg u,v =(0,1)$.
 An element $F\in R$ is bihomogeneous if $F\in R_{(i,j)}$ for some $(i,j)\in \mathbb{N}$. 
 If $F$ is bihomogeneous, we say that its bidegree is $\deg F =(i,j)$. Suppose that
 $I=(F_{1},\ldots,F_{n})\subset R$ is an ideal. If each $F_{i}$ is bihomogeneous, then
 we say that $I$ is a bihomogeneous ideal.
 
  For $U=\{f_{0},\ldots,f_{3}\}\subset H^{0}(a,1)$ we let $I_{U}:=\langle f_{0},\ldots,f_{3} \rangle$
 be the bihomogeneous ideal of $R$ generated by the elements of $U$.

 \subsection{Description of sets of points in $\Q$} We follow the notation and definitions from
 Giuffrida, Maggioni and Ragusa \cite{OnThePost} and from Guardo and Van Tuyl \cite{guardovanacm} to describe sets
 of points in $\Q$.
 A point in $\Q$ is a pair $A\times B$, where $A\in \pp^{1}$ and 
 $B\in \pp^{1}$. Let $h$ be a non-zero linear form in the variables $s,t$
 that vanishes at $A$ and let $l$ be a non-zero linear form in the variables
 $u,v$ that vanishes at $B$. The ideal of $R$ that corresponds to $P$
 is denoted by $I_{P}$ and $I_{P}=\langle h,l \rangle$. We think
 of the form $h$ as a $(1,0)$ line and of $l$ as a $(0,1)$ line in
 $\Q$. These lines are members of the two different rulings of $\Q$
 and a point in $\Q$ is uniquely determined by their intersection.
 If $X=\{P_{1},\ldots,P_{r}\}$ is a set of $r$ distinct points and
 $I_{P_{i}}=\langle h_{i},l_{i}\rangle$, then $I_{X}$, the ideal
 corresponding to $X$, is given by $I_{X}=\bigcap_{i=1}^{r}I_{P_{i}}$.

 There are two projections $\pi_{i}:\Q\to \pp^{1}$ defined by
 $\pi_{1}(A\times B)=A$ and $\pi_{2}(A\times B)=B$. We can
 visualize points in $\Q$ as lying inside a grid of 
 $(1,0)$ and $(0,1)$ lines.

 From now on, $(1,0)$ lines will be drawn horizontally and $(0,1)$ lines will 
 be drawn vertically. Throughout this paper, $X$ will denote a finite set
 of points in $\Q$ and $I_{X}$ will denote its corresponding defining ideal
 in $R$.
 
 As with sets of points in $\pp^{n}$, we use Hilbert functions
 to study sets of points in $\Q$. Since the ring $R$ is bigraded, the
 Hilbert function of $X$ takes the shape of a matrix.
 \begin{definition}
 Let $X$ be a finite set of points in $\Q$. The bigraded Hilbert function of 
 $X$, $H_{X}:\mathbb{Z}\times \mathbb{Z}\to \mathbb{N}$ is defined
 by
 \[H_{X}(i,j)=\dim_{k}R_{(i,j)}-\dim_{k}(I_{X})_{(i,j)}.\]
 \end{definition}
 The bigraded Hilbert function of $X$ has similar properties to the
 Hilbert function of a set of points in projective space. When $\Q$
 is considered as a subvariety of $\pp^{3}$ by the Segre embedding, 
 $X$ becomes a subscheme of $\pp^{3}$, in this case, 
 $H_{X}(i)=H_{X}(i,i)=\deg X$ for all $i\gg0$. The work of Guardo and Van Tuyl \cite{guardovanacm} provides
 a thorough introduction to the study of bigraded Hilbert functions
 of points in $\Q$.

 \subsection{Combinatorial description of sets of points in $\Q$}
 Using the description of a single point in $\Q$ as the
 intersection of a $(1,0)$ line and a $(0,1)$ line, 
 a set of points in $\Q$ can be visualized as markings of some
 intersection points inside a rectangular grid as in Example \ref{pointgrid}.
 We describe sets of points in a
 combinatorial way using partitions. Let $h_{1},\ldots,h_{q}$ be the 
 horizontal lines in $\Q$ that contain points of $X$. The ordering of
 these lines doesn't play any role so we may arrange them in such a way that
 \[|X\cap h_{1}|\geq |X\cap h_{2}| \geq \ldots \geq |X\cap h_{q}|.\]
 Let $\alpha_{i}=|X\cap h_{i}|$ and associate to $X$ the tuple $\alpha_{X}=
 (\alpha_{1},\ldots,\alpha_{q})$. Note that $\alpha_{X}$ is a partition of $|X|$.
 An analogous process yields a partition  $\beta_{X}$ from $X$, 
 which depends on the points on the vertical rulings of $X$.
 For the next theorem, we use the notion of the conjugate of a partition.
 \begin{definition}
 The conjugate of a partition $\lambda$ is the tuple $\lambda^{\ast}=
 (\lambda_{1}^{\ast},\ldots,\lambda_{\lambda_{1}}^{\ast})$ where $\lambda_{i}^{\ast}
 =\# \{\lambda_{j}\in \lambda | \lambda_{j}\geq i\}$.
 \end{definition}
 For a set $X$ of points in $\Q$ we usually have $\alpha_{X}^{\ast}\neq \beta_{X}$.
 Theorem~\ref{partitionh} describes the Hilbert function $H_{X}(i,j)$ for $(i,j)\gg(0,0)$ in terms of the partitions $\alpha, \beta$. It was first formulated by Van Tuyl \cite{van2002border}, here we use the statement
  from Guardo and Van Tuyl \cite{guardovanacm}.
 
 \begin{theorem}[Guardo-Van Tuyl \cite{guardovanacm}] \label{partitionh}
              Let $X\subset \Q$ be any set of points with associated tuples $\alpha_{X}=
              (\alpha_{1},\ldots,\alpha_{h})$ and $\beta_{X}=(\beta_{1},\ldots,\beta_{\nu})$
               and let $h=|\pi_{1}(X)|$ and $\nu=|\pi_{2}(X)|$.
              \begin{enumerate}
                 \item For all $j\in \mathbb{N}$, if $i\geq h-1$, then
                         \[H_{X}(i,j)=\alpha_{1}^{\ast}+\cdots+\alpha_{j+1}^{\ast}\]
                       where $\alpha_{X}^{\ast}=(\alpha_{1}^{\ast},\ldots,\alpha_{\alpha_{1}}^{\ast})$
              is the conjugate of $\alpha_{X}$ and where we make the convention that $\alpha_{l}^{\ast}=0$
                 if $l>\alpha_{1}$.
            \item For all $i\in \mathbb{N}$, if $j\geq \nu -1$, then
            \[H_{X}(i,j)=\beta_{1}^{\ast}+\cdots+\beta_{j+1}^{\ast}\]
             where $\beta_{X}^{\ast}=(\beta_{1}^{\ast},\ldots,\beta_{\beta_{1}}^{\ast})$
             is the conjugate of $\beta_{X}$ and where we make the convention that $\beta_{l}^{\ast}=0$
             if $l>\beta_{1}$.
              \end{enumerate}
 \end{theorem}
 
 Thus if we know the values of $H_{X}(i,j)$ for $(i,j)\gg 0$ we are able to determine $\alpha$ and $\beta$.
 This in turn gives us information about the vertical an horizontal rulings that contain $X$. We will
 use this theorem to give a geometric description of a generic set of $r$ points in $\Q$ in subsection 
 \ref{diagonal}.
   \begin{example} \label{pointgrid}
   Let $X$ be the set of points in $\Q$ on the left below. Then $\alpha_{X}=(4,4,3,2)$, $\beta_{X}=(3,2,2,2,2,2)$ and  $\alpha_{X}^{\ast}
 =(4,4,3,2)$, $\beta_{X}^{\ast}=(6,6,1)$. The previous theorem implies that
 $H_{X}$ has  the following form.
   
   \begin{minipage}{0.5\textwidth}
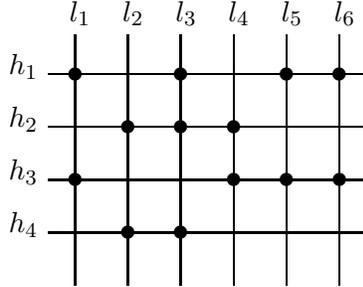
   \begin{center} 
   \begin{picture}(120,120) 
\put(0,20){\line(1,0){120}}
\put(0,40){\line(1,0){120}}
\put(0,60){\line(1,0){120}}

\put(8,100){$l_{1}$}
\put(28,100){$l_{2}$}
\put(48,100){$l_{3}$}
\put(68,100){$l_{4}$}
\put(88,100){$l_{5}$}
\put(108,100){$l_{6}$}

\put(-15,80){$h_{1}$}
\put(-15,60){$h_{2}$}
\put(-15,40){$h_{3}$}
\put(-15,20){$h_{4}$}

\put(0,80){\line(1,0){120}}
\put(10,0){\line(0,1){95}}

\put(30,0){\line(0,1){95}}

\put(50,0){\line(0,1){95}}
\put(70,0){\line(0,1){95}}
\put(90,0){\line(0,1){95}}
\put(110,0){\line(0,1){95}}
\put(70,0){\line(0,1){95}}
\put(10,80){\circle*{5}}
\put(50,80){\circle*{5}}
\put(90,80){\circle*{5}}
\put(110,80){\circle*{5}}
\put(30,60){\circle*{5}}
\put(50,60){\circle*{5}}
\put(70,60){\circle*{5}}
\put(10,40){\circle*{5}}
\put(70,40){\circle*{5}}
\put(90,40){\circle*{5}}
\put(110,40){\circle*{5}}
\put(30,20){\circle*{5}}
\put(50,20){\circle*{5}}

\end{picture}
 \end{center}
 \end{minipage}
 \begin{minipage}{0.3\textwidth}
  \[ H_{X}=
   \begin{array}{c||c|c|c|c|c|c|c} 
     & 0 & 1 & 2 & 3 & 4 & 5 & 6 \\ \hline \hline
    0&   &   &   &   &   & 6 & 6 \\ \hline
    1&   &   &   &   &   & 12& 12 \\ \hline
    2&   &   &   &   &   & 13& 13 \\ \hline
    3& 4 & 8 & 11& 13& 13& 13& 13 \\ \hline
    4& 4 & 8 & 11& 13& 13& 13& 13 \\ 
   \end{array}
   \]
 \end{minipage}
 
   The blank entries cannot in general be deduced from the information contained in 
   $\alpha_{X}$ and $\beta_{X}$. In Examples \ref{4genericpoints} and \ref{nongeneric} we illustrate this fact with two sets of points $X_{1},X_{2}$ in $\Q$ whose Hilbert function is different but for which
   $\alpha_{X_{1}}=\alpha_{X_{2}}$ and $\beta_{X_{1}}=\beta_{X_{2}}$.
   \end{example}

 \section{Ideal of a generic set of points in $\Q$} \label{idealOfPoints}
 In this section we  describe minimal generators of $I_{X}$ in bidegree $(a,1)$ when $X$ is
 a generic set of points of $\Q$. For this purpose we use the work of Van Tuyl~\cite{van2005defining}
 for the case of $\Q$ . Using this
 description in Section~\ref{syzygiesOfIdeal} will allow us to find syzygies of $I_{U}=\langle 
 f_{0},\ldots,f_{3}\rangle$.

\subsection{Description of $I_{X}$}
\begin{definition}
Let $X$ be a set of $r$ points in $\pp^{1}\times \pp^{1}$. The set  $X$  is said to be  generic if its bigraded
Hilbert function,  is determined by
\[H_{X}(i,j)= \min \{ (i+1)(j+1),r\}.\]
\end{definition}

Being a generic set of points is a property of the Hilbert function. Note that from the definition
of a generic set of points in $\Q$ we cannot immediately describe the bidegrees of the minimal
generators of $I_{X}$ nor the distribution of the points of $X$ in families of $(1,0)$ and $(0,1)$ 
lines. We focus on describing these two aspects for the rest of this section.

 If $\mathcal{E}$ is any subset of $\mathbb{N}^{2}$ and 
$\underline{a}=(a_{0},a_{1})
\in \mathbb{N}^{2}$ is any tuple, then
$\mathcal{E}+\underline{a}$ denotes the set $\{\underline{e}+\underline{a}: \underline{e}
\in \mathcal{E}\}$. Let $e_{1}=(1,0)$ and $e_{2}=(0,1)$.

For every $i\geq 0$ 
\[j(i):=\min\{t\in \mathbb{N}| H_{X}(i,t)=H_{X}(i,t+1)\},\]
and for every $j\geq 0$ 
\[i(j):=\min\{t\in \mathbb{N}| H_{X}(t,j)=H_{X}(t+1,j)\}.\]
\begin{theorem}[Van Tuyl\cite{van2005defining}] 
               \label{thm3}
        Let $X$ be a finite set of points in $\Q$. Fix $e_{1}=(1,0)$ and $j\in \mathbb{N}$. Set
        $\underline{i}= (i(j),j)$. Then
        \[(I_{X})_{\underline{i}+(q+1)e_{1}}=R_{e_{1}}(I_{X})_{\underline{i}+qe_{1}}\;\;\; \forall q\in 
         \mathbb{N}.\]
      In particular, if there exists $\underline{l}\in \mathbb{N}^{2}$ and $t\in \{1,2\}$ such that
       $H_{X}(\underline{l})=H_{X}(\underline{l}-e_{t})=H_{X}(\underline{l}-2e_{t})$. Then $I_{X}$
       has no minimal generators of bidegree $\underline{l}$.
\end{theorem}
Theorem~\ref{thm3} gives a way to exclude bidegrees that will not show up as bidegrees
of minimal generators for $I_{X}$. It also gives us a description of the higher degree pieces of
$I_{X}$ once the bigraded Hilbert function has stabilized in one of the coordinates. 

Define the $k[s,t]$-module  $M$ by
\[M=\bigoplus_{i=0}^{\infty}(I_{X})_{(i,1)}.\]
    Following the notation of Theorem~\ref{thm3}, fix $j=1$, then $i(1)$ is the value of $i$ at which
    the bigraded Hilbert function $H_{X}$ of $X$ stabilizes for the column $j=1$. So for all
    $q\in \mathbb{N}$,
    \[H_{X}(i(1),1)=H_{X}(i(1)+q,1).\]
    Recall that $r$ is the number of points in $X$, we want to find $i(1)$ and understand
    the generators of $M$ as a $k[s,t]$ module. 
    
    \begin{proposition} \label{gens}
Let $X$ be a generic set of $r$ points in $\Q$
with associated bihomogeneous ideal $I_{X}$. Then the $k[s,t]$-module
$M$ has two minimal generators in bidegrees $(k,1),(k,1)$ if $r=2k$
and two minimal generators in bidegrees $(k,1),(k+1,1)$ if $r=2k+1$.
\end{proposition}
    
    \begin{proof}
    Suppose $r=2k$. Then 
    \[H_{X}(i,1)=\min\{r,(i+1)2\}.\]
    For all $0\leq i<k-1$, $H_{X}(i,1)=(i+1)2$ and for $i\geq k-1$, $H_{X}(i,1)=r$. Thus
    $i(1)=k-1$ and we have
    \[r=H_{X}(k-1,1)=H_{X}(k,1)=H_{X}(k+1,1).\]
    Using Theorem~\ref{thm3} we know that $I_{X}$ has no minimal generators of degree $(k+1,1)$ or higher.
    Thus  possible minimal generators of $I_{X}$ are in bidegrees $(k-1,1),(k,1)$. But we
    have
    \begin{equation} \label{dimcount}
    H_{X}(k-1,1)=r=\dim R_{(k-1,1)}-\dim (I_{X})_{(k-1,1)},
    \end{equation}
    so $\dim (I_{X})_{(k-1,1)}=0$ and $I_{X}$ does not have generators in bidegree $(k-1,1)$.
    A similar equation to (\ref{dimcount}) shows that $\dim (I_{X})_{(k,1)}=2$. Therefore $M$ has two 
    generators
    of bidegree $(k,1)$. Moreover, using the last part of Theorem~\ref{thm3} we have
    \begin{equation}\label{generators}
      (I_{X})_{i(1)+(q+1)e_{1}}=R_{e_{1}}(I_{X})_{i(1)+qe_{1}}  \;\;\; \forall q\in \mathbb{N}.
    \end{equation}
    Since $R_{e_{1}}=\{s,t\}$ this shows $M$ has two generators in bidegree $(k,1)$ as a 
    $k[s,t]$ module. The proof for the case when $r$ is odd
    is done in a similar way.
   \end{proof}

\subsection{Geometric description of $X$ in terms of rulings}\label{diagonal}

Let $X$ be a set of generic points in $\Q$. Using the Hilbert function of $X$,
we see that for $(i,j)$ with $0\leq j$ and $i\geq r-1$, $H_{X}(i,j)=r$ and
for $(i,j)$ with $0\leq i$ and $j\geq r-1$, $H_{X}(i,j)=r$. Using Theorem
\ref{partitionh}, we know that for $j=0$ and $i\gg0$, $H_{X}(i,0)=
\alpha_{1}^{\ast}=r$. Since $X$ consists exactly of $r$ points, we conclude
that $\alpha_{i}^{\ast}=0$ for all $i>0$ and $\alpha_{X}=(1,\ldots,1)$ has 
$r$ parts. Similarly, $\beta_{X}=(1,\ldots,1)$ has $r$ parts. Using the
definition of $\alpha_{X}$, we know that the number of parts of $\alpha_{X}$ is equal
to the number of horizontal lines in $\Q$ that contain points of $X$. Moreover,
each part $\alpha_{i}$ of $\alpha_{X}$ corresponds to the number of points in $X$
that lie on the horizontal line $h_{i}$. From $\alpha_{X}=(1,\ldots,1)$, we
conclude that there are $r$ distinct horizontal lines in $\Q$ each containing exactly
one point of $X$. We have a similar statement for $\beta_{X}$ and vertical lines.
After reordering of the vertical and horizontal lines that contain points of $X$,
we can describe $X$ as the set of diagonal points inside an $r\times r$ grid of lines.
\begin{example}
\label{4genericpoints}
Let $I_{X_{1}}=(s,u)\cap(t,v)\cap(s-3t,u-v)\cap(s+t,u+5v)$. Using \texttt{Macaulay2} we
compute the bigraded Hilbert function $H_{X_{1}}$ on the right and conclude $X_{1}$ is generic. The distribution
of $X_{1}$ on the rulings of $\Q$ is illustrated on the left.
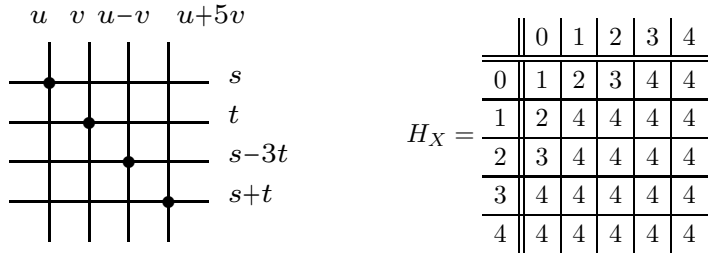
\begin{center}
\scalebox{1.5}{
\begin{picture}(200,70)

\put(0,10){\line(1,0){50}}
\put(0,20){\line(1,0){50}}
\put(0,30){\line(1,0){50}}
\put(0,40){\line(1,0){50}}

\put(5,55){\tiny $u$}
\put(15,55){\tiny$v$}
\put(22,55){\tiny$u\scalebox{0.75}[1.0]{$-$}v$}
\put(42,55){\tiny$u\scalebox{0.75}[1.0]{$+$}5v$}

\put(55,40){\tiny$s$}
\put(55,30){\tiny$t$}
\put(55,20){\tiny$s\scalebox{0.75}[1.0]{$-$}3t$}
\put(55,10){\tiny$s\scalebox{0.75}[1.0]{$+$}t$}

\put(10,40){\circle*{3}}
\put(20,30){\circle*{3}}
\put(30,20){\circle*{3}}
\put(40,10){\circle*{3}}

\put(10,0){\line(0,1){50}}
\put(20,0){\line(0,1){50}}
\put(30,0){\line(0,1){50}}
\put(40,0){\line(0,1){50}}

\put(100,25){\scalebox{0.6}{$H_{X}=\begin{array}{c||c|c|c|c|c} 
     & 0 & 1 & 2 & 3 & 4  \\ \hline \hline
    0& 1 & 2 & 3  & 4  & 4    \\ \hline
    1& 2  & 4  & 4  & 4  &  4  \\ \hline
    2& 3  & 4  & 4  & 4  &  4  \\ \hline
    3& 4 & 4 & 4& 4& 4 \\ \hline
    4& 4 & 4 & 4& 4& 4 \\ 
   \end{array}$}}
\end{picture} }
 \end{center}
 Using Proposition~\ref{gens}, we know $I_{X_{1}}$ has two generators $g_{1},g_{2}$ in bidegree $(2,1)$.
 A \texttt{Macaulay 2} computation yields $g_{1}=6t^{2}u+7s^{2}v-23stv, g_{2}=2stu-3s^{2}v+7stv$.

\end{example}
\begin{example} \label{nongeneric}
Let $I_{X_{2}}=(s,u)\cap(t,v)\cap(s-t,u-v)\cap(s+t,u+v)$. Using \texttt{Macaulay2} we compute
the bigraded Hilbert function  $H_{X_{2}}$ on the right and conclude that $X_{2}$ is not generic.
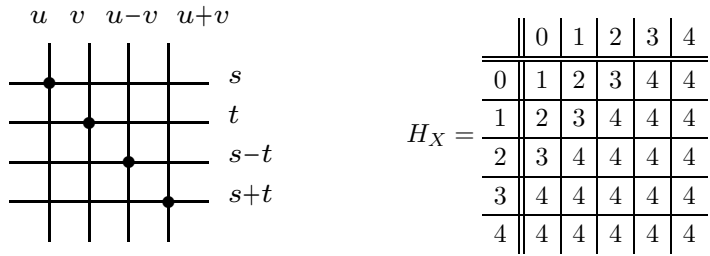
\begin{center}
\scalebox{1.5}{
\begin{picture}(200,70)
\put(0,10){\line(1,0){50}}
\put(0,20){\line(1,0){50}}
\put(0,30){\line(1,0){50}}
\put(0,40){\line(1,0){50}}

\put(5,55){\tiny $u$}
\put(15,55){\tiny$v$}
\put(24,55){\tiny$u\scalebox{0.75}[1.0]{$-$}v$}
\put(42,55){\tiny$u\scalebox{0.75}[1.0]{$+$}v$}

\put(55,40){\tiny$s$}
\put(55,30){\tiny$t$}
\put(55,20){\tiny$s\scalebox{0.75}[1.0]{$-$}t$}
\put(55,10){\tiny$s\scalebox{0.75}[1.0]{$+$}t$}

\put(10,40){\circle*{3}}
\put(20,30){\circle*{3}}
\put(30,20){\circle*{3}}
\put(40,10){\circle*{3}}

\put(10,0){\line(0,1){50}}
\put(20,0){\line(0,1){50}}
\put(30,0){\line(0,1){50}}
\put(40,0){\line(0,1){50}}

\put(100,25){\scalebox{0.6}{$H_{X}=\begin{array}{c||c|c|c|c|c} 
     & 0 & 1 & 2 & 3 & 4  \\ \hline \hline
    0& 1 & 2 & 3  & 4  & 4    \\ \hline
    1& 2  & 3  & 4  & 4  &  4  \\ \hline
    2& 3  & 4  & 4  & 4  &  4  \\ \hline
    3& 4 & 4 & 4& 4& 4 \\ \hline
    4& 4 & 4 & 4& 4& 4 \\ 
   \end{array}$}}
\end{picture} }
 \end{center}
The ideal $I_{X_{2}}$ has a minimal generator of bidegree $(1,1)$ whereas the generic set of points
$X_{1}$ from Example~\ref{4genericpoints} has no generator of such bidegreee. However we have
$\alpha_{X_{1}}=\alpha_{X_{2}}$ and $\beta_{X_{1}}=\beta_{X_{2}}$. In this case the module $M$
from Proposition~\ref{gens} has two generators of bidegrees $(1,1)$ and $(3,1)$.
\end{example}

\section{Syzygies of the ideal $I_{U}$} \label{syzygiesOfIdeal}
 As was highlighted in the introduction, the syzygies of $I_{U}$ determine the complex $\zc$
 which in turn determine the implicit equation of $X_{U}$. Let $\mathrm{Syz}(I_{U})_{(-,0)}$ denote 
 the $k[s,t]$-module of syzygies of $I_{U}=\langle f_{0},\ldots,f_{3}\rangle \subset R$ of degree zero in $u,v$. The following proposition states that
 $\mathrm{Syz}(I_{U})_{(-,0)}$ has rank two and gives the bidegrees of its minimal generators.
 
 \begin{proposition} \label{syzygyprop}
                    Let $X$ be a generic set of $r$ points in $\Q$ and let $U=
                    \{f_{0},\ldots,f_{3}\}$ be a generic $4$-dimensional vector space of $(I_{X})_{(a,1)}$
                    with $a\geq k$. Then the $k[s,t]$-module $\mathrm{Syz}(I_{U})_{(-,0)}$ has rank two.
                    Moreover, $\mathrm{Syz}(I_{U})_{(-,0)}$ has minimal generators of bidegree $(a-k,0)$ if $r=2k$
                    and of bidegrees $(a-k,0),(a-k+1,0)$ if $r=2k+1$.
 \end{proposition}
 
 \begin{remark} \label{generic}
 If $\mathbf{b}= \{b_1,\ldots,b_q \}$ is a basis for $(I_X)_{(a,1)}$, we make the convention that a generic choice of $U$
 is given by a matrix $C$ of size $4\times q$ with $\mathbf{f}=C\mathbf{b}$ all of whose maximal minors are nonzero.
 \end{remark}
 The proof of Proposition \ref{syzygyprop} will be done in four steps. The steps that will follow are done for the case
 $r=2k$. If $r=2k+1$, similar steps work by changing to $(a-k,0)$ and $(a-k+1,0)$.
 \newcounter{Stepcount}[section]
 \begin{step} \label{step1}
 Use the description of a basis for $(I_{X})_{(a,1)}$ to show that
 the generators of the ideal $I_{U}$ may be written in a simpler and more convenient form.

 From Proposition~\ref{gens} we know $(I_{X})_{(k,1)}$
 is generated by two elements $g_{1},g_{2}$. Thus by equation (\ref{generators}) in the proof of Proposition \ref{gens}, a basis for $(I_{X})_{(a,1)}$ ,
 is given by
 \[\mathbf{b}=
   \{ s^{j}g_{1},s^{j-1}tg_{1},\ldots, t^{j}g_{1},s^{j}g_{2},\ldots,st^{j-1}g_{2},t^{j}g_{2}
   \},\]
   where $j=a-k$.
 Following our convention from Remark \ref{generic}, a generic choice of $U$ is a coefficient matrix $C$ such that every maximal minor
 of $C$ is nonzero with $\mathbf{f}=C\mathbf{b}$. Using the basis $\mathbf{b}$ for $(I_X)_{(a,1)}$
 \[
 \begin{array}{llll}( f_{0} & f_{1}&f_{2}&f_{3}) \end{array}^{T}
    =  \begin{array}{lllllll}
       (s^{j}g_{1} & s^{j-1}tg_{1}& \cdots & t^{j}g_{1}& s^{j}g_{2}& \cdots & t^{j}g_{2})
       \end{array} C^{T}.
 \]
 The first $j+1$ elements of the basis share the common factor $g_{1}$ and the last $j+1$ elements
 share the factor $g_{2}$, thus we may write
 \begin{equation}
\begin{array}{llll} (f_{0}& f_{1}&f_{2}& f_{3})\end{array}
 =\begin{array}{ll}( g_{1} & g_{2} )\end{array}
 \QP  \label{QP}
 \end{equation}
 where $\deg Q_{i}=\deg P_{i}=(j,0)$ are forms in the variables $s,t$. 
 Denote the rightmost matrix in (\ref{QP}) by $QP$. 
 \end{step}
  Using the notation in the first step we let
    \begin{equation}\label{A}
    (Q_{0}\;\; Q_{1}\;\;Q_{2}\;\;Q_{3})=(s^{j}\;\; s^{j-1}t \; \cdots \; t^{j}) A^{T},
 \end{equation}
and
   \begin{equation} \label{B}
    (P_{0}\;\; P_{1}\;\;P_{2}\;\;P_{3})=(s^{j}\;\; s^{j-1}t \; \cdots \; t^{j}) B^{T}.
   \end{equation}
 Denote the matrix of coefficients of $Q_{i}$'s by $A$ and the matrix of coefficients
 of $P_{i}$'s by $B$. Observe that C is the block matrix formed by $A$ and $B$, so $C=(A|B)$. 
 
 \begin{step} \label{step2}
 Reduce the problem to the study of the kernel of the matrix $QP$. 
 \end{step}
 
 \begin{lemma} \label{lemma2}
 The ideal $I_{U}$ has a syzygy $L$ of bidegree $(\alpha,0)$ if and only if $L$ is an element in the kernel
 of the matrix
 \[ QP=\left(
     \begin{array}{llll} 
      Q_{0} & Q_{1} & Q_{2} & Q_{3} \\
      P_{0} & P_{1} & P_{2} & P_{3} \\
     \end{array} \right)
 \]
 over the ring $k[s,t]$.
 \end{lemma}
 \begin{proof}
 Suppose that $L=(s_{0},s_{1},s_{2},s_{3})$ is a syzygy of $(f_{0},f_{1},f_{2},f_{3})$ of degreee
 $(\alpha,0)$. Then
 \[
 s_{0}f_{0}+s_{1}f_{1}+s_{2}f_{2}+s_{3}f_{3}=0.
 \]
 Using Step~\ref{step1}, $f_{i}=Q_{i}g_{1}+P_{i}g_{2}$. We may substitute
 this in the previous equation and factor $g_{1}$ and $g_{2}$ as follows, 
 \[(s_{0}Q_{0}+s_{1}Q_{1}+s_{2}Q_{2}+s_{3}Q_{3})\cdot g_{1}+
 (s_{0}P_{0}+s_{1}P_{1}+s_{2}P_{2}+s_{3}P_{3})\cdot g_{2}=0\]
The elements $g_{1},g_{2}$ are minimal generators of $I_{X}$ and thus form a complete intersection. Using
the fact that the only syzygies of a complete intersection are Koszul we obtain
 \[   \left(
      \begin{array}{c}
      s_{0}Q_{0}+s_{1}Q_{1}+s_{2}Q_{2}+s_{3}Q_{3} \\
      s_{0}P_{0}+s_{1}P_{1}+s_{2}P_{2}+s_{3}P_{3}
      \end{array}
      \right)
      = q\left( \begin{array}{r}-g_{2}\\ g_{1}\end{array} \right).
 \]
 Note that $\deg s_{0}Q_{0}+s_{1}Q_{1}+s_{2}Q_{2}+s_{3}Q_{3}=\deg s_{0}P_{0}+s_{1}P_{1}+s_{2}P_{2}+s_{3}P_{3} =(\alpha+j,0)$. But $\deg g_{i}=(k,1)$, hence $q=0$ and
 \[   \left(
      \begin{array}{c}
      s_{0}Q_{0}+s_{1}Q_{1}+s_{2}Q_{2}+s_{3}Q_{3} \\
      s_{0}P_{0}+s_{1}P_{1}+s_{2}P_{2}+s_{3}P_{3}
      \end{array}
      \right)
      = 0
 \]
 so $L\in \ker QP$. Verifying the other direction is straightforward from the previous calculations.
 \end{proof}
 
\begin{step}\label{step3}
When $QP$ is a matrix whose entries are polynomials in $k[s,t]$ of degree $j$, with coefficients from a generic matrix $C$ as in 
equations~(\ref{A}) and (\ref{B}), the kernel of $QP$ has rank two and its minimal generators are of degree $j$.

A quick check reveals that the assumption $\gcd(f_{0},\ldots,f_{3})=1$ implies the matrix $QP$ has rank two. Then it can be proved
that if the coefficients of $QP$ are chosen generically following Remark~\ref{generic}, the two minimal generators of $\ker QP$
have degree $j=a-k$.  If the coefficients of $QP$ are not chosen generically, we can expect the degrees of the 
two minimal generators
of $\ker QP$ to be less than or equal to $j$.
\end{step}

\begin{step}
 Using Step~\ref{step2}, we know that the syzygies of $I_U$ in bidegree $(\alpha,0)$ are in one-to-one correspondence with elements in the
 kernel of $QP$. Following Step \ref{step3},  the kernel of $QP$ is generated by two elements $K1,K2$ of degree $j$ in $s,t$, hence $I_U$ has two minimal
 syzygies of bidegree $(j,0)$.
 \end{step}

 \begin{remark}\label{free}
  Note that the two elements $K1,K2$ in the kernel of $QP$ generate a free module. Indeed,
  $QP$ fits into a sequence
  \begin{equation} \label{res}
   \xymatrix{  0 \ar[r] & S^{2} \ar[r] & S^{4} \ar[r]^{QP}& S^{2} \ar[r] & \mathrm{coker}\; QP \ar[r]& 0 \\
   }
  \end{equation}
  where the leftmost nonzero map sends $S^{2}$ to the two syzygies $K1,K2$ and $S=k[s,t]$. Then by Hilbert's syzygy theorem over $S$, the leftmost map in the sequence (\ref{res}) is injective. 
 \end{remark}


\section{Applications to implicitization of tensor product surfaces}
\label{implicitization}

 We now focus on obtaining the implicit equation of $X_{U}$, when $U=\{f_{0},\ldots,f_{3}\}
 \subseteq (I_{X})_{(a,1)}$ is a generic $4$-dimensional subspace. First we give a definition of 
 $\zc$ and explain
 how to set up the complex $\zc_{\nu}$ with an example. Next, we give a proof of the main theorem and apply these 
 techniques 
 for the basepoint free case. We also illustrate the advantages of the syzygy method
 over Gr\"{o}bner bases. All of the 
 examples presented in this section were computed using the software \texttt{Macaulay2}.
 Throughout this section, the implicit equation of a parameterized surface in $\pp^{3}$ will be denoted by 
 $H$ and it is an element in $S=k[X,Y,Z,W]$. 
 
 \subsection{Implicitization via syzygies}
 Given $U=\{f_{0},\ldots,f_{3}\}$, the complex $\zc$ has $\zc_{i}=\ker(d_{i}^{f})\otimes S$ where $d_{i}^{f}$
 is the differential of the Koszul complex associated with the sequence $f=(f_{0},\ldots,f_{3})$ and $S=k[X,Y,Z,W]$ is 
 the coordinate ring of $\pp^{3}$. The differential 
 $\partial_{i}:\zc_{i}\to\zc_{i-1}$ for the complex $\zc$ is defined to be the differential of the
 Koszul complex associated with the sequence $(X,Y,Z,W)$.  We use the results by Botbol \cite{botbol}
 to compute the implicit equation of $X_{U}$ via the approximation complex $\zc$. Recall that $\zc_{1}=\mathrm{Syz}(f_{0},\ldots,f_{3})\otimes S$.
  \begin{theorem}[Botbol \cite{botbol}] \label{bot}
  Let $\lambda: \Q\dashrightarrow \pp^{3}$ be a rational map with finitely many local complete
  intersection basepoints given by 4 bihomogeneous polynomials $f_{0},\ldots,f_{3}\in R_{(a,b)}$
  without common factors. Then for $\nu=(2a-1,b-1)$ (equiv. $\nu=(a-1,2b-1)$)
  \[\det{\zc_{\nu}}= H^{\deg{\lambda}} \in k[X,Y,Z,W],\]
  where $H$ is the irreducible implicit equation of $X_{U}$.
  \end{theorem}
  Since $U\subseteq(I_{X})_{(a,1)}$, $X$ is the set of basepoints of $X_{U}$ and
  we are able to use Theorem~\ref{bot}. To illustrate how the computation of the implicit
  equation works we have the following example.
  \begin{example}\label{ex1}
   Let $X$ be the set of points in $\Q$ associated to the ideal $I_{X}=(s,u)\cap (t,v)=(uv,sv,tu,st)$.
  Computing the Hilbert function of $X$ in \texttt{Macaulay2} shows that $X$ is generic. Take
  $(a,1)=(3,1)$, then 
  \[(I_{X})_{(3,1)}=\{s^{3}v,t^{3}u,s^{2}tu,s^{2}tv,st^{2}u,st^{2}v\}.\]
  We take $U$ to be the generic $4$\scalebox{0.75}[1.0]{$-$}dimensional subspace given by
  \[U=\{ s^{3}v\scalebox{0.75}[1.0]{$-$}st^{2}u \scalebox{0.75}[1.0]{$+$}st^{2}v,t^{3}u\scalebox{0.75}[1.0]{$+$}st^{2}u\scalebox{0.75}[1.0]{$+$}st^{2}v,s^{2}tu\scalebox{0.75}[1.0]{$+$}st^{2}u\scalebox{0.75}[1.0]{$-$}3st^{2}v,s^{2}tv\scalebox{0.75}[1.0]{$-$}5st^{2}u
  \scalebox{0.75}[1.0]{$+$}st^{2}v\}.\]
 To compute the implicit
  equation of $X_{U}$, we need to set up the complex $\zc_{\nu}$ and then find $\det \zc_{\nu}$.
  The complex $\zc_{\nu}$,
  \[
   \xymatrix{  \mathcal{Z}_{\nu} : & 0\ar[r]&(\mathcal{Z}_{2})_{\nu} \ar[r]^{d_{2}} & (\mathcal{Z}_{1})_{\nu}    
    \ar[r]^{d_{1}} & (\mathcal{Z}_{0})_{\nu} \ar[r]& 0 
    }
  \] only has three nonzero terms and it is exact.
  Following Theorem~\ref{bot}, $\nu=(5,0)$. The ideal $I_{U}$ has two syzygies in degree $(2,0)$ and
  three syzygies in bidegree $(1,1)$. The syzygies in degree $(2,0)$ span a free $R$-module and determine  a 
  basis for the syzygies of bidegree $(5,0)$ by multiplying  by $\{s^{3},s^{2}t,st^{2},t^{3}\}$ . We proceed to apply the map $\partial_{1}$ to this  
  basis and obtain the matrix $d_{1}$.
   Using \texttt{Macaulay2} and $d_{2}=\ker d_{1}$ we obtain
         \[ d_{1}=\resizebox{1.1\hsize}{!}{$
       \left(\begin{array}{llllllll}
               107Y          & 0             & 0              &0              &107W                  &0                    & 0                     &0        \\             
              \scalebox{0.75}[1.0]{$-$}228Y\scalebox{0.75}[1.0]{$-$}107Z\scalebox{0.75}[1.0]{$-$}52W &107Y           & 0              &0              &\scalebox{0.75}[1.0]{$-$}107X\scalebox{0.75}[1.0]{$+$}1082Y\scalebox{0.75}[1.0]{$+$}535Z\scalebox{0.75}[1.0]{$+$}335W &107W                 & 0                     &0         \\           
              \scalebox{0.75}[1.0]{$-$}55X\scalebox{0.75}[1.0]{$-$}32Z\scalebox{0.75}[1.0]{$-$}41W   &\scalebox{0.75}[1.0]{$-$}228Y\scalebox{0.75}[1.0]{$-$}107Z\scalebox{0.75}[1.0]{$-$}52W & 107Y           &0              &\scalebox{0.75}[1.0]{$-$}442X\scalebox{0.75}[1.0]{$-$}49Z\scalebox{0.75}[1.0]{$+$}295W        &\scalebox{0.75}[1.0]{$-$}107X\scalebox{0.75}[1.0]{$+$}1082Y\scalebox{0.75}[1.0]{$+$}535Z\scalebox{0.75}[1.0]{$+$}335W& 107W                  &0        \\             
                 0           &   \scalebox{0.75}[1.0]{$-$}55X\scalebox{0.75}[1.0]{$-$}32Z\scalebox{0.75}[1.0]{$-$}41W& \scalebox{0.75}[1.0]{$-$}228Y\scalebox{0.75}[1.0]{$-$}107Z\scalebox{0.75}[1.0]{$-$}52W &107Y           &0                     &\scalebox{0.75}[1.0]{$-$}442X\scalebox{0.75}[1.0]{$-$}49Z\scalebox{0.75}[1.0]{$+$}295W       & \scalebox{0.75}[1.0]{$-$}107X\scalebox{0.75}[1.0]{$+$}1082Y\scalebox{0.75}[1.0]{$+$}535Z\scalebox{0.75}[1.0]{$+$}335W &107W     \\             
                 0           &   0           &   \scalebox{0.75}[1.0]{$-$}55X\scalebox{0.75}[1.0]{$-$}32Z\scalebox{0.75}[1.0]{$-$}41W &  \scalebox{0.75}[1.0]{$-$}228Y\scalebox{0.75}[1.0]{$-$}107Z\scalebox{0.75}[1.0]{$-$}52W& 0                   &  0                  &   \scalebox{0.75}[1.0]{$-$}442X\scalebox{0.75}[1.0]{$-$}49Z\scalebox{0.75}[1.0]{$+$}295W      &  \scalebox{0.75}[1.0]{$-$}107X\scalebox{0.75}[1.0]{$+$}1082Y\scalebox{0.75}[1.0]{$+$}535Z\scalebox{0.75}[1.0]{$+$}335W \\
                 0           &   0           &   0            &  \scalebox{0.75}[1.0]{$-$}55X\scalebox{0.75}[1.0]{$-$}32Z\scalebox{0.75}[1.0]{$-$}41W  & 0                   &  0                  &   0                   &  \scalebox{0.75}[1.0]{$-$}442X\scalebox{0.75}[1.0]{$-$}49Z\scalebox{0.75}[1.0]{$+$}295W        \\
       \end{array}
        \right),$
        }
       \]
      \[ d_{2}= \resizebox{0.5\hsize}{!}{$
        \left( \begin{array}{ll}
                   \scalebox{0.75}[1.0]{$-$}1/107W                            &  0                                  \\
                 1/107X\scalebox{0.75}[1.0]{$-$}1082/11449Y\scalebox{0.75}[1.0]{$-$}5/107Z\scalebox{0.75}[1.0]{$-$}335/11449W &\scalebox{0.75}[1.0]{$-$}1/107W                              \\
                 442/11449X\scalebox{0.75}[1.0]{$+$}49/11449Z\scalebox{0.75}[1.0]{$-$}295/11449W      &1/107X\scalebox{0.75}[1.0]{$-$}1082/11449Y\scalebox{0.75}[1.0]{$-$}5/107Z\scalebox{0.75}[1.0]{$-$}335/11449W \\
                 0                                    &442/11449X\scalebox{0.75}[1.0]{$+$}49/11449Z\scalebox{0.75}[1.0]{$-$}295/11449W      \\
                 1/107Y                               &0                                    \\
                 \scalebox{0.75}[1.0]{$-$}228/11449Y\scalebox{0.75}[1.0]{$-$}1/107Z\scalebox{0.75}[1.0]{$-$}52/11449W         &1/107Y                               \\
                 \scalebox{0.75}[1.0]{$-$}55/11449X\scalebox{0.75}[1.0]{$-$}32/11449Z\scalebox{0.75}[1.0]{$-$}41/11449W       &\scalebox{0.75}[1.0]{$-$}228/11449Y\scalebox{0.75}[1.0]{$-$}1/107Z\scalebox{0.75}[1.0]{$-$}52/11449W         \\
                 0                                    &\scalebox{0.75}[1.0]{$-$}55/11449X\scalebox{0.75}[1.0]{$-$}32/11449Z\scalebox{0.75}[1.0]{$-$}41/11449W       \\
       \end{array} \right)$}
       .\]
   The determinant of $\zc_{\nu}$ is computed by $\det \zc_{\nu}=\det M_{1}/\det M_{2}$, where $M_{1}$ is a
   maximal nonzero minor of $d_{1}$ and $M_{2}$ is the complimentary maximal nonzero minor of $d_{2}$,
   
   \[\det{\zc_{\nu}}=  \resizebox{0.8\hsize}{!}{$ 
        \begin{array}{l}
          - 8831798120631365X^{3}Y + 623043212873630840X^{2}Y^{2}  - 2432437780569525764XY^{3} \\ + 154155021741929280X^{2}YZ - 
                                                                         
      2181293557648299312XY^{2}Z - 694982223019864504Y^{3}Z - \\ 516419322835463088XYZ^{2}  - 679406142698023733Y^{2}Z^{2}  - 
                              
      167804164291995935YZ^{3}\\  + 29064644724259583X^{2}YW - 2253232567794532976XY^{2}W + 347491111509932252Y^{3}W + \\ 8831798120631365X^{2}ZW -
                                                   
      1142192364219107259XYZW - 288398353175526028Y^{2}ZW - \\ 109996031138772455XZ^{2}W - 277318460987824861YZ^{2}W -
                         
      33560832858399187Z^{3} W \\+ 8831798120631365X^{2}W^{2}  - 417342605736743957XYW^{2}  + 335768906731639713Y^{2}W^{2} \\
       - 103733483380506578XZW^{2}  +
                          
      6583704053561563YZW^{2}  - 20232846603628218Z^{2}W^{2}  - \\ 20232846603628218XW^{3}  + 65676462387967787YW^{3}  + 3693297395900389ZW^{3}  + \\
                       
      3693297395900389W^{4}.
        \end{array}
        $}
   \]
  \end{example}

  As was mentioned in the introduction, the matrix $d_{1}$ representing the map 
  $(\zc_{\nu})_{1}\to(\zc_{\nu})_{0}$ is sufficient to perform the elementary geometric operations
  that arise in CAGD. For this reason the growth of the size of the coefficients for the implicit equation
   in practice does not affect the
  speed of the computations.

  \subsection{Proof of main theorem}
    To prove  Theorem~\ref{mainth} we need two lemmas, the first one states the dimension of 
    $(\zc_{1})_{\nu}$
  and the second one provides a basis for $(\zc_{1})_{\nu}$.
  \begin{lemma}\label{lemma1}
  Let $\zc$ be the approximation complex associated to $(f_{0},\ldots,f_{3})$, with differential
  on the variables $(X,Y,Z,W)$. If $U=\{f_{0},\ldots,f_{3}\}$ has $r$ basepoints, each of multiplicity one,
  then \[\dim_{k} (\zc_{1})_{\nu}=2ab+r\]
  where $\nu=(2a-1,b-1)$.
  \end{lemma}
  
  \begin{proof}
       From the proof of Lemma 7.3 in Botbol \cite{botbol}, we know that $\zc_{\nu} $ is acyclic and that
       $(\zc_{3})_{\nu}=0$. Hence $\zc_{\nu}$ has the form
        \[
 \xymatrix{
 \zc_{\nu}:& 0 \ar[r] & (\mathcal{Z}_{2})_{\nu} \ar[r] & (\mathcal{Z}_{1})_{\nu} \ar[r] &(\mathcal{Z}_{0})_{\nu}
 \ar[r]& 0.
 }
 \] Using Lemma 7.3 and 7.4 from Botbol~\cite{botbol}, and the fact that $I_{U}$ is a local complete intersection,
 Botbol observes that 
 \[\dim_{k}(\mathcal{Z}_{2})_{\nu}=\sum_{x\in X}e_{x} \]
 where $e_{x}$ denotes the multiplicity of the basepoint $x$ of $U$ and $X$ is the set of basepoints of
 $U$. 
 Since $U$ has $r$ basepoints each of multiplicity one, $\dim_{k} (\zc_{2})_{\nu}=r$. Also, $(\zc_{0})_{\nu}
 =R_{(2a-1,b-1)}$, thus $\dim_{k} (\zc_{0})_{\nu}=2ab$. It follows that
 $\dim_{k}(\zc_{1})_{\nu}=2ab+r$.  \end{proof}
  
  \begin{lemma}\label{lemma2}
       Let $U=\{f_{0},\ldots,f_{3}\}$ be a generic $4$-dimensional vector subspace of $(I_{X})_{(a,1)}$
       and $\zc$ the approximation complex associated to $(f_{0},\ldots,f_{3})$. Then a basis for $(
       \zc_{1})_{\nu}$, $\nu=(2a-1,b-1)$, is obtained by bumping up the syzygies of $I_{U}$ in bidegree
       $(a-k,0)$ if $r=2k$ and the syzygies in bidegrees $(a-k,0),(a-k+1,0)$ if $r=2k+1$.
  \end{lemma}
  
  \begin{proof}
  We have $\zc_{1}=\ker(d_{1}^{f})=\mathrm{Syz}(f_{0},\ldots,f_{3})$, thus
  $(\zc_{1})_{\nu}=(\mathrm{Syz}(f_{0},\ldots,f_{3}))_{\nu}$. If $r$ is even then $I_{U}$ has two syzygies
  $S_{1},S_{2}$ of bidegree $(a-k,0)$ by Proposition \ref{syzygyprop}. Using Remark \ref{free} at the end of 
  Section \ref{syzygiesOfIdeal}, we know that $S_{1},S_{2}$ span a free module. Hence the set of
  syzygies $\mathcal{S}=\{m\cdot S_{i}| m\in R_{(a+k-1,0)},i=1,2\}$ is an independent subset of $(\mathrm{Syz}(
  f_{0},\ldots,f_{3}))_{\nu}$. Note that $\dim_{k} R_{(a+k-1,0)}=a+k$, thus $\dim_{k} \mathrm{Span}\;
  \mathcal{S}= 2a+r$. From Lemma \ref{lemma1} $\dim (\zc_{1})_{\nu}=2a+r$, hence a basis for 
  $(\zc_{1})_{\nu}$ is determined by bumping up $S_{1},S_{2}$. An analogous argument works in the case
  $r$ is odd.
  \end{proof}
  We are now ready to prove Theorem \ref{mainth}.
  \begin{theorem*}
  Let $(I_{X})\subseteq H^{0}(a,1)$ be the $k$-vector space of forms of bidegree $(a,1)$ that vanish at
  a generic set $X$ of $r$ points in $\Q$. Take $U=\{f_{0},\ldots,f_{3}\}$ to be a generic $4$-dimensional
  vector subspace of $(I_{X})_{(a,1)}$ and $\lambda_{U}:\Q\dashrightarrow \pp^{3}$ the rational map determined
  by $U$. Then the first map of the approximation complex $\zc$ in bidegree $\nu=(2a-1,0)$ is determined
  by the syzygies of $(f_{0},\ldots,f_{3})$ in bidegree $(a-k,0)$ if $r=2k$ and bidegrees
  $(a-k,0),(a-k+1,0)$ if $r=2k+1$. These syzygies completely determine the complex $\zc_{\nu}$
  from which we compute the implicit equation.
  \end{theorem*}
  \begin{proof}
  We consider the case $r=2k$, the case $r=2k+1$ is done similarly. Let $S_{1},S_{2}$ be the two 
  minimal syzygies of
  $I_{U}$ in bidegree $(a-k,0)$. By Lemma \ref{lemma2}, $(\zc_{1})_{\nu}$ has a basis given by the elements
  of the form $m\cdot S_{i}$ where $i=1,2$ and $m\in R_{(a+k-1,0)}$. Denote the matrix of the first map of
  the complex $\zc_{\nu}$ by $d_{1}$. Then $d_{1}$ is obtained by applying the
   Koszul differential on the sequence $(X,Y,Z,W)$ to all the elements $\{m\cdot S_{i}\}$. Using
   the proof of Lemma \ref{lemma1}, we know $\zc_{\nu}$ is exact, hence $d_{2}=\ker
   d_{1}$.  Therefore the approximation complex $\zc_{\nu}$ that determines the implicit equation
   of $X_{U}$ only depends on the syzygies of $I_{U}$ in degree $(j,0)$.
  \end{proof}

  \begin{corollary}
  Let $X$ be a set of $r$ points in $\Q$ with $r=2k$. Take $U$ to be any $4$-dimensional vector
  subspace of $(I_{X})_{(k+1,1)}$, then $X_{U}$ is smooth.
  \end{corollary}
  \begin{proof}
  From Proposition~\ref{gens}, $(I_{X})_{(k+1,1)}=\{sg_{1},tg_{1},sg_{2},tg_{2}\}$ so 
  $\dim_{k}(I_{X})_{(k+1,1)}=4$. Thus up to a change of coordinates,
  any choice of $U$ is equivalent to the choice $U=\{sg_{1},tg_{1},sg_{2},tg_{2}\}$. For this $U$,
  it follows that $X_{U}=\mathbb{V}(XW-YZ)$, hence $X_{U}$ is smooth.
  \end{proof}
  
  The results in Theorem~\ref{mainth} do not generalize immediately to tensor product surfaces of more
  general bidegree $(a,b)$ with $b>1$. One of the advantages of the condition $b=1$ is that
  the calculation of the syzygies of $I_{U}$ is reduced to finding the kernel of the matrix
  $QP$ over the polynomial ring $k[s,t]$ which has two fewer variables than $R$. This allows us
  to show in Remark~\ref{free} that the syzygies of $I_{U}$ are free. For more general bidegree,
  the syzygies of $I_{U}$ are not free and computing a basis for $\mathrm{Syz}(I_{U})_{\nu}$ is
  more difficult because of the possible relations between the generators of $\mathrm{Syz}(I_{U})_{\nu}$.
  It would be interesting to know if the syzygies of $I_{U}$ for $U\subset R_{(a,b)}, b>1$ can
  also be calculated and understood from the generators of the ideal $I_{X}$ of the basepoints of
  $U$ as is the case for $b=1$.
 \subsection{Tensor product surfaces without basepoints}
 The main theorem in this paper allows us to describe the syzygies that determine the implicit equation of a map given
 by $4$ generically chosen forms of bidegree $(a,1)$ that vanish at a generic set of points in $\Q$. The techniques that
 we used to prove the main theorem depended on understanding the generators of the $k[s,t]$-module
 $M=\bigoplus_{i=0}^{\infty}(I_{X})_{(i,1)}$ and using the description of $(I_{X})_{(a,1)}$ given by 
 (\ref{generators}) from Section \ref{idealOfPoints}. For the case of basepoints, $M=\langle g_{1},g_{2}\rangle$. If $X=\emptyset$ then 
 $M=\langle u,v \rangle$. In the proof of the main theorem, we may substitute $g_{1},g_{2}$ for
 the complete intersection $u,v$ and the proof will still be valid. We obtain
 \begin{theorem}\label{bpfree}
 Let $U=\{f_{0},\ldots,f_{3}\}$ be a basepoint free, generic, $4$-dimensional 
 vector subspace of $R_{(a,1)}=H^{0}(a,1)$ and $\lambda_{U}:\Q\to \pp^{3}$ the regular map determined by $U$.
 Then the first map of the
     approximation complex $\mathcal{Z}$ in bidegree $\nu=(2a-1,0)$, 
     is determined by two syzygies of $(f_{0},\ldots,f_{3})$
     in bidegree $(a,0)$.
 \end{theorem}
 A careful study of the implicitization of basepoint free tensor product surfaces of bidegree $(2,1)$ 
 using syzygies was done by Schenck, Seceleanu and Validashti \cite{SSTPS}. Theorem~\ref{bpfree} recovers their results for the case
 that $I_{U}$ has no linear syzygies and extends them to bidegree $(a,1)$. 
\subsection{Overview of implicitization methods and examples}
As mentioned in the introduction, we may use other methods like Gr\"{o}bner bases and resultants to
find implicit equations of parameterized surfaces. One way to use Gr\"{o}bner bases is to compute
the elimination ideal
\begin{equation}\label{elimination}J=\langle X-f_{0},Y-f_{1},Z-f_{2},W-f_{3}\rangle\cap S. \end{equation}
Resultants are used for basepoint free parameterizations of the form $\lambda: \pp^{2}\to \pp^{3}$.
In this case $H$ is computed by 
\[\mathrm{Res}(f_{0}-Xf_{3},f_{1}-Yf_{3},f_{2}-Zf_{3})=H(X,Y,Z,1)^{\deg \lambda}.\]
In the following examples we will compute $H$ using Gr\"{o}bner bases and the complex $\zc_{\nu}$. For
further examples that show the advantages of implicitization using syzygies over other methods, we refer
the reader to the work of Botbol and Dickenstein \cite{botdick}(Section 5). The \texttt{Macaulay2} code to
perform the  examples that follow is available at \url{https://github.com/emduart2}.
\begin{example} 
We let $X$ be the set of points in $\Q$  in Example~\ref{ex1}, where $I_{X}=(s,u)\cap(t,v)$. We fix
the bidegree $(a,b)=(8,1)$ and let $U\subset (I_{X})_{(8,1)}$ be given by
\[\begin{array}{ll}
   U= \{&  3s^{6}t^{2}u+s^{8}v+7s^{4}t^{4}v-s^{3}t^{5}v,
      \;\;t^{8}u+5s^{6}t^{2}v+s^{2}t^{6}v,\\
      &s^{7}tu+11s^{5}t^{3}u+s^{2}t^{6}u-st^{7}u,
     \;\; -s^{4}t^{4}u+s^{3}t^{5}u+st^{7}u+s^{7}tv+19s^{5}t^{3}v+st^{7}v\}.
      \end{array}\]
The coefficient matrix of $U$ with respect to a monomial basis for $(I_{X})_{(8,1)}$ is generic. We use
two algorithms in \texttt{Macaulay2} to compute the first map $d_{1}:(\zc_{\nu})_{1}\to (\zc_{\nu})_{0}$ of $\zc_{\nu}$, we refer to these as Algorithm 1 and Algorithm 2.
Algorithm 1 computes the minimal syzygies of $I_{U}$, then it selects the two syzygies in bidegree $(7,0)$.  
Using Theorem~\ref{mainth}, we know these syzygies are free. Then it
bumps up these two syzygies to bidegree $\nu=(2a-1,b-1)=(15,0)$ by multiplying times the monomials in a basis for $R_{(8,0)}$. Finally
it finds $d_{1}$ by applying the Koszul map in the variables $(X,Y,Z,W)$.
With this method it takes $0.133$ seconds to compute $d_{1}$. Algorithm 2  directly computes
a basis for the syzygies of $I_{U}$ and then uses the command $\texttt{super basis(\{3*a-1,2*b-1\}, image syz Iu) }$ to find a basis of the syzygies in bidegree $\nu=(2a-1,b-1)$. Finally it proceeds to set up $d_{1}$ as before. With this method the computation of $d_{1}$ takes
$1.67$ seconds. Notice the speed boost in Algorithm 1 that is obtained from knowing the structure of the syzygies that 
determine $d_{1}$, in particular we know the exact degrees of the syzygies of $I_{U}$ and that they span a free 
module. If we use the \texttt{eliminate} command in \texttt{Macaulay2} to compute the elimination ideal
$J$ from equation~(\ref{elimination}), the computation of the implicit equation takes $149.178$ seconds. For this example, $d_{1}$ is a 
$16 \times 18 $ matrix, the implicit equation $H$ of $X_{U}$ has degree $14$ and it contains $115$ terms.
\end{example}

\begin{example}
Take $U$ to be a random generic $4$-dimensional vector subspace of $R_{(20,1)}$, then $U$ is basepoint free. This example can be generated in
\texttt{Macaualay2} by finding a basis of $R_{(20,1)}$ using \texttt{super basis(\{20,1\},R)} and then multiplying this basis times a random matrix $C$ of coefficients of the correct size. For this choice of
$U$, Algorithm 1 takes $0.751$ seconds to compute $d_{1}:(\zc_{\nu})_{1}\to(\zc_{\nu})_{0}$ and
Algorithm 2 takes $180$ seconds. The \texttt{eliminate} ideal command did not finish the computation
in at least $120$ minutes and therefore was aborted. In this example, $d_{1}$ is a $40\times 40$ matrix and
$H$ has degree $40$.
\end{example}

%
 
  \noindent \textbf{Acknoledgements} The  author would like to thank  Hal Schenck for 
  useful discussions and suggestions as well as the referee for the helpful comments that improved the
  exposition of this work.  Evidence for this work was provided by many computations done
 using \texttt{Macaulay2}, by Dan Grayson and Mike Stillman. \texttt{Macaulay2} is freely available
 at \url{http://www.math.uiuc.edu/Macaulay2/} and scripts to perform the computations are available at
 \url{https://github.com/emduart2}.
 

\bibliographystyle{acm}	
\bibliography{refs}

\end{document}